\documentclass[12pt]{amsart}
\usepackage{amssymb} 
\usepackage[mathscr]{eucal} 
\usepackage{amsmath, amsthm} 
\usepackage{epsfig}
\usepackage{amscd}
\usepackage{mathrsfs}
\usepackage{tikz-cd}
\usepackage{hyperref}  
\usepackage{mathtools}
\usepackage{fullpage}

\def\R{{\mathbb R}}

\def\Hom{\mathrm{Hom}}

\def\Diff{\mathrm{Diff}}

\def\St{\mathbf{St}}
\def\Grpds{\mathbf{Grpds}}
\def\MM{{\mathcal M}}

\def\i{{\mathfrak i}}

\def\d{{\overline{d}}}
\def\i{{\overline{\iota}}}

\theoremstyle{plain}
 \newtheorem{teo}{Theorem}[section]
 \newtheorem{prop}[teo]{Proposition}

\theoremstyle{definition}
 \newtheorem{exm}[teo]{Example}
 \newtheorem{dfn}[teo]{Definition}
\theoremstyle{remark}
 \newtheorem{rem}[teo]{Remark}

\hypersetup{allcolors=black, colorlinks=true} 

\numberwithin{equation}{section}

\begin{document}

\title[Equivariant cohomology for differentiable stacks]
{Equivariant cohomology for differentiable stacks}

\author[Luis Alejandro Barbosa-Torres]
{Luis Alejandro Barbosa-Torres}
\address{School of Mathematics and Actuarial Science\\
Pure Mathematics Group\\
University of Leicester\\
University Road, Leicester LE1 7RH, England, UK}
\email{labarbosat@gmail.com}

\author[Frank Neumann]
{Frank Neumann}
\address{School of Mathematics and Actuarial Science\\
Pure Mathematics Group\\
University of Leicester\\
University Road, Leicester LE1 7RH, England, UK}
\email{fn8@le.ac.uk}




\begin{abstract}
{We construct and analyse models of equivariant cohomology for differentiable stacks with Lie group actions extending classical results for smooth manifolds due to Borel, Cartan and Getzler. We also derive various spectral sequences for the equivariant cohomology of a differentiable stack generalising among others Bott's spectral sequence which converges to the cohomology of the classifying space of a Lie group.} 
\end{abstract}

\maketitle

\section*{Introduction}

\noindent An important principle in geometry and physics is to exploit symmetry whenever possible. A common manifestation of such symmetry 
is for example given by an action of a Lie group $G$ on a smooth manifold $X$. Equivariant cohomology is a way of exploiting this symmetry and provides an important algebraic invariant to study Lie group actions on smooth manifolds. A first approach when looking for a notion of cohomology in this framework is to use the singular cohomology of the quotient space $X/G$, however even though singular cohomology is well defined, the quotient and its singular cohomology may loose a lot of geometric information. In order to overcome this obstruction, one can construct alternative and better models for equivariant cohomology using a weaker, but more adequate notion of a quotient for Lie group actions on smooth manifolds by employing homotopy theory or using appropriate equivariant versions of the de Rham complex of differential forms. There are basically three important models for equivariant cophomology. Firstly, the Borel model provides a good topological model for equivariant cohomology by considering the homotopy quotient space instead, also called Borel construction, $EG\times_G X$ of $X$ (see \cite{borel1960coh}, \cite{guillemin2013supersymmetry}). Secondly, the Cartan model employs the notion of equivariant differential forms on $X$, but restricted to compact Lie group actions (see \cite{cartan1951coh}). Both models provide the same equivariant cohomology groups (see \cite{berline2003heat}). And finally, the Getzler model (see \cite{getzler1994equivariant}), which also employs an appropriate complex of equivariant differential forms, provides a generalisation and an alternative to the Cartan model in the more general case of arbitrary, not necessarily compact Lie group actions on smooth manifolds. This model is also essential for many important applications in global analysis, differential geometry and mathematical physics.

In this article we provide an alternative and extension to the classical constructions of equivariant cohomology using the more general framework of Lie group actions on differentiable stacks. This approach relies on the construction of a good quotient stack or stacky quotient $\MM/G$ for a general action of a Lie group $G$ on a given differentiable stack $\MM$. In the special situation of an action of a Lie group $G$ on a smooth manifold $X$ we recover the quotient stack $[X/G]$. We will show that this stacky quotient $\MM/G$ is again a differentiable stack and has a homotopy type given by a good homotopy quotient constructed via the fat geometric realisation of the nerve of the Lie groupoid associated to the differentiable stack $\MM/G$. In more detail, we can understand this homotopy type of $\MM/G$ as being constructed from the bisimplicial smooth manifold given by $G^\bullet \times X_\bullet$, where $G^p$ is the $p$-fold cartesian product of the Lie group $G$ and $X_n$ describes the $n$-th component of the simplicial smooth manifold given by the nerve associated to the Lie groupoid constructed from a given $G$-atlas $X\rightarrow \MM$ of the original differentiable $G$-stack $\MM$. These constructions are in fact independent of a particular choice of an atlas and are therefore stacky by nature. In the special situation of an action of a Lie group $G$ on a smooth manifold $X$ we obtain the classical quotient stack $[X/G]$ and its homotopy type simply recovers the homotopy quotient $EG\times_G X$. Consequently, we obtain a Borel model for equivariant cohomology of general differentiable $G$-stacks which extends and generalises the classical Borel model for smooth $G$-manifolds.  Furthermore using simplicial techniques, we extend the Cartan and Getzler models for equivariant cohomology based on particular complexes of equivariant differential forms for differentiable stacks with compact or non-compact Lie group actions. For example, the stacky Cartan model is induced from constructions of adequate complexes of differential forms for simplicial smooth manifolds based on Meinreken's work (see \cite{meinrenken2005witten}, \cite{suzuki2015simplicial}, \cite{suzuki2016equivariant}) while for the stacky Getzler model we also use simplicial techniques from \cite{kubel2015equivariant} applied to the simplicial smooth manifold build from the nerve of the associated Lie groupoid of the given differentiable $G$-stack. Group actions on stacks were first studied systematically by Romagny \cite{romagny2003group} in the context of algebraic geometry where general actions of flat group schemes on algebraic stacks were considered. More recently general actions of topological groups on topological stacks were also studied by Ginot-Noohi \cite{ginot2012group} in their general approach to equivariant string topology \cite{behrend2007string}. Since differentiable stacks generalise orbifolds (see \cite{lerman2010orbst}, \cite{satake1956generalization}, \cite{thurston1979geometry}), the models of equivariant cohomology for Lie group actions on differentiable stacks provided here also establish good models for equivariant orbifold cohomology which we aim to explore in future work. Equivariant Chen-Ruan orbifold cohomology for example has many important applications in symplectic geometry, in particular when considering Hamiltonian torus actions on orbifolds (compare for example \cite{holmmatsumura2012orb}, \cite{lermanmakin2012orb}).

This article is structured as follows. In the first section we recall the definition and basic properties of general Lie group actions on differentiable stacks. We start by defining the notion of an action of a Lie group $G$ on a differentiable stack $\MM$ and then describe the associated $2$-category  $G$--$\St$ of $G$-stacks. This is very much in the flavour of \cite{romagny2003group} and \cite{ginot2012group}, but in the context of smooth manifolds and Lie groups. The second section features the construction of the quotient stack $\MM/G$ for a differentiable stack $\MM$ with an action of a Lie group $G$. We then introduce the notion of a differentiable $G$-stack $\MM$ using an appropriate version of a $G$-atlas given as a smooth manifold with $G$-action and relate it to the simplicial $G$-manifold constructed from the associated Lie groupoid of $\MM$. Finally we show that the quotient stack $\MM/G$ is in fact a differentiable stack and we explicitly describe its homotopy type given by the fat geometric realisation of the simplicial nerve of the Lie groupoid associated to $\MM/G$. In the third section we recall and discuss several cohomology theories for differentiable stacks, namely de Rham cohomology, sheaf cohomology and hypercohomology, following in parts the expositions in \cite{behrend2004cohomology}, \cite{behrend2005deRham} and \cite{heinloth2005notes}. In the fourth section we introduce the concepts of equivariant cohomology for differentiable $G$-stacks and derive and describe the Borel, Cartan and Getzler models in this general context. We analyse several fundamental properties of equivariant cohomology, in particular concerning the effect of restricting the acting Lie group. Finally in the fifth and last section we derive several spectral sequences that all converge to the equivariant cohomology of a differentiable $G$-stack. They relate the cohomology of the simplicial nerve of the Lie groupoid associated to a differentiable $G$-stack $\MM$ with the equivariant cohomology of $\MM$. We will then analyse these in particular situations and discuss their homological properties. As special cases we obtain generalisations of spectral sequences previously constructed for equivariant cohomology of smooth $G$-manifolds, including the celebrated Bott spectral sequence converging to the cohomology of the classifying space of a Lie group (see \cite{felder2008gerbe}, \cite{stasheff1978continuous} and \cite{bott1973chern}).

\section{Group actions on differentiable stacks} \label{GroupActStack}

\noindent The general notion of a group action on a stack was first developed and applied by Rogmany \cite{romagny2003group} in the context of group scheme actions on algebraic stacks. More recently, Ginot and Noohi \cite{ginot2012group} studied topological group actions on topological stacks in their approach to equivariant string topology. In this section we will recall and analyse the main definitions and constructions within the framework of differentiable stacks.
For us here, a {\it stack} will always mean a pseudo-functor $\MM:\Diff^{op} \rightarrow \Grpds$, where $\Diff$ is the category of smooth manifolds and smooth maps endowed with the big site of local diffeomorphisms, and where $\Grpds$ is the category of discrete groupoids (see also \cite{heinloth2005notes} or \cite{Barbosaphdthesis, neumann2009algebraic}). We can interpret any smooth manifold $X$ in $\Diff$  also as the stack given by $\underline{X}=\text{Hom}(- , X)$, which is the stack represented by $X$. A {\it differentiable stack} is a stack $\MM$, together with a smooth manifold $X$ and a morphism of stacks $X\xrightarrow{p} \MM$, called an {\it atlas}, which is representable and has local sections (see \cite{heinloth2005notes}). For more details about the theory of differentiable stacks and their basic properties we will refer to \cite{Barbosaphdthesis, heinloth2005notes} and \cite{behrend2011differentiable}.

\begin{dfn} 
Let $G$ be a Lie group and $\mathcal{M}$ a differentiable stack with atlas $X\rightarrow \mathcal{M}$. A $G$\textit{-action on} $\mathcal{M}$ is a morphism of stacks $\mu: G\times \mathcal{M} \rightarrow \mathcal{M}$ together with $2$-morphisms $\alpha$ and $\beta$, such that for each $T \in \Diff$, the following diagrams  

\begin{center}
 \begin{tikzcd}[column sep=large]
G\times G\times \mathcal{M} (T) \arrow[r, ""{below, name=W, inner sep=1pt}, "m\times id_{\mathcal{M}(T)}"{above}]{} \arrow[d,""{right,name=X,inner sep=1pt},"id_G \times\mu_T"{left}]
& G\times \mathcal{M} (T) \arrow{d}{\mu_T} \\
G\times \mathcal{M} (T) \arrow{r}{\mu_T} \arrow[ur,Rightarrow,"\alpha"]
& \mathcal{M} (T)
\end{tikzcd}
\end{center}
and

\begin{center}
\begin{tikzcd}[column sep=large]

G\times \mathcal{M} (T)
\arrow[r,"\mu_T"{above}]
&\mathcal{M} (T) \\
\mathcal{M}(T)\arrow[u, "e \times id_{\mathcal{M}(T)}"{left}]{} \arrow[ur," id_{\mathcal{M}(T)}"{swap}]{}  &\arrow[ul, Rightarrow,shorten <= 5em, shorten >= 0.02em, pos=0.85, "\beta"]
\end{tikzcd}
\end{center}
are $2$-commutative, that is, for every $T\in\Diff$ the following holds, where the dot $\cdot$ denotes the action $\mu$:
\begin{enumerate}
	\item $(g\cdot \alpha_{h,k}^x) \alpha_{g,hk}^x = \alpha_{g,h}^{k\cdot x}\alpha_{gh,k}^x$, for all $g,h,k \in G$ and $x\in \MM(T)$.
	
	\item $(g\cdot \beta^x )\alpha_{g,e}^x = 1_{g\cdot x} = \beta^{g\cdot x} \alpha^{x}_{e,g} $ for every $g\in G$, $x\in \MM(T)$ and $e$ the identity in $G$.
\end{enumerate}
And where $\alpha_{g,h}^{x}:g\cdot(h\cdot x) \rightarrow (gh)\cdot x$ and $\beta^x:x\rightarrow e\cdot x$ in $\MM(T)$.
 
\end{dfn}

\begin{dfn}
The $4$-tuple $(\mathcal{M},\mu, \alpha, \beta)$ is called a $G$-\textit{stack}, where $\mu$ is a $G$ action on $\mathcal{M}$. 
\end{dfn}

\begin{dfn} 
A \textit{morphism of $G$-stacks} between $(\mathcal{M},\mu, \alpha, \beta)$ and $(\mathcal{N}, \nu, \gamma, \delta)$ is a morphism of stacks $F:\MM \rightarrow \mathcal{N}$ together with a $2$-morphism $\sigma$ and the following 2-commutative diagram

\begin{center}
 \begin{tikzcd}[column sep=large]
G\times \mathcal{M} \arrow{r}{\mu} \arrow{d}[swap]{id_G \times f}
& \mathcal{M} \arrow{d}{f}  \arrow[dl,Rightarrow,"\sigma "]\\
G\times \mathcal{N}  \arrow{r}{\nu}
& \mathcal{N}
\end{tikzcd}
\end{center}
such that, for every $T\in \Diff$
\begin{enumerate}
	\item $\sigma_g^{h\cdot x}( g\cdot \sigma_{h}^x)\gamma_{g,h}^{F(x)} = F(\alpha_{g,h}^x) \sigma_{gh}^{x}$, for every $g,h\in G$ and $x\in \MM(T)$.
	
	\item $F(\beta^x)\sigma_e^x = \delta^{F(x)}$, for every object $x\in \MM(T)$ and $e$ the identity element of $G$.
\end{enumerate}
where $\sigma_g^x:F(g\cdot x)\rightarrow g\cdot F(x)$ in $\mathcal{N}(T)$. 

\end{dfn}

\begin{dfn}

A \textit{2-morphism of $G$-stacks} between 1-morphism of $G$-stacks, $(F,\sigma)$ and $(F',\sigma')$, is a 2-morphism of stacks $\phi: F\Rightarrow F'$  such that
\begin{enumerate}
	\item[$(3)$] $(\sigma_g^x)(g\cdot \phi_x)= (\phi_{g\cdot x})(\sigma_g^{'x})$ for every $g\in G$ and $x\in \MM(T)$.
\end{enumerate}

\noindent Here $\phi_x:F(x)\rightarrow F'(x)$ denotes the $2$-morphism $\phi$ when applied to $x\in \MM(T)$.

\end{dfn}

\begin{rem}
In this way, we obtain a 2-category of $G$-stacks denoted by $G$--$\St$. 
\end{rem}




\begin{rem}
A $G$-action of a Lie group $G$ on a smooth manifold $M$ coincides with the one above for differentiable stacks $\MM$, where the diagrams are now strictly commutative instead. Similarly, the notion of $G$-equivariant smooth maps in $\Diff$ in this special example coincides with the one of morphism of $G$-stacks.
\end{rem}

\section{Quotient stacks} \label{GeneralQSt}

\noindent The concept of quotient stacks for group actions on stacks was first developed and studied by Romagny \cite{romagny2003group} for algebraic stacks and by Ginot-Noohi \cite{ginot2012group} for topological stacks. This generalises the classical notions of quotient stacks arising from group actions on schemes, manifolds or topological spaces. Here again we will be working entirely in the differentiable setting of differentiable stacks with Lie group actions.

\begin{dfn}
Let $G$ be a Lie group acting on a differentiable stack $\MM$. Consider the pseudo-functor
\[ \MM/G : \Diff^{op} \rightarrow \Grpds \]
such that for each $T\in \Diff$, an element in $\mathcal{M}/ G(T)$ is a triple $t=(p, f, \sigma)$ such that $p:E\rightarrow T$ is a principal $G$-bundle and  $(f,\sigma): E\rightarrow \mathcal{M}$ is an equivariant morphism. The arrows in $\mathcal{M}/ G(T)$ are pairs $(u, \alpha)$ with a $G$-morphism $u:E\rightarrow E'$ and a $2$-commutative diagram of $G$-stacks given by
\begin{center}
\begin{tikzcd}
E \arrow{rr}{u}
\arrow[rd, ""{name=U}]{}[swap]{(f, \sigma)} 
&& E' \arrow{dl}{(f', \sigma')}[swap,""{name=V}]{} 
\arrow[Rightarrow, from=U, to=V, "\alpha"]\\
&\mathcal{M}
\end{tikzcd}
\end{center}
If there is a smooth map $T\xrightarrow{h} S$, then there exists a morphism $\MM/G(S)\rightarrow \MM/G(T)$ given by the pullback as in the following commutative diagram
\begin{center}
 \begin{tikzcd}
T\times_{S}E \arrow{r} \arrow{d}{h^\ast}
& E \arrow{r}{(f, \sigma)} \arrow{d}{p} & \MM\\
T \arrow{r}{h}
& S &
\end{tikzcd}
\end{center}
where $\MM/G(h)=h^\ast$.

\end{dfn}

\begin{prop}
Let $G$ be a Lie group with an action on a differentiable stack $\MM$. The pseudo-functor $\MM/G$ is a stack.  
\end{prop}

\begin{proof}
Since it is possible to glue principal $G$-bundles, the gluing conditions in the definition of a stack hold. Therefore the quotient $\MM/G$ is indeed a stack.
\qedhere

\end{proof}

\begin{exm}
Let $M$ be a smooth manifold with an action of a Lie group $G$. In this special case we obtain the usual quotient stack $[M/G]$ for Lie group actions on smooth manifolds, defined for each $T\in \Diff$ via the groupoid of sections as
$$ [M/G](T)=\left\langle  (E\xrightarrow{p} T, E\xrightarrow{f} M) : p \text{ is a principal $G$-bundle}, f \text{ is an equivariant map} \right\rangle  .$$
This is the same as the stacky quotient $M/G=\text{Hom}(- ,M)/G$ defined as above (see also \cite{heinloth2005notes}). 
\end{exm}

Another way to consider this stack, is by defining a prestack $\mathcal{P}$ such that for $T\in \Diff $ we have $\mathcal{P}(T)= \mathcal{M}(T)$ and morphisms between $x$ and $y$ in $\MM(T)$ are pairs $(g,\varphi)$ with $g\in G$ and $\varphi:g.x \rightarrow y$ a morphism in $\mathcal{M}(T)$. If we use stackification, we get the stack $(\mathcal{M}/ G)^\ast = \tilde{\mathcal{P}}$ associated to $\mathcal{P}$
and we have (see \cite[Sec 4]{ginot2012group}):

\begin{prop}
The stacks $\mathcal{M}/ G $ and $(\mathcal{M} / G)^*$ are isomorphic.

\end{prop}

\begin{proof}

Consider the morphism $\Phi: \mathcal{P}\rightarrow \mathcal{M}/ G $ such that for each $T\xrightarrow{x} \mathcal{M}$, $\Phi(x)$ is the following principal $G$-bundle

\begin{center}
 \begin{tikzcd}[column sep=large]
G\times T \arrow{d}{p_2} \arrow{r}{\mu \circ (id \times x)} & \mathcal{M} \\
 T 
\end{tikzcd}
\end{center}
where $\mu$ is the action on $\MM$.

We recall that a morphism $(g,\varphi): x \rightarrow y$ in $\mathcal{P}$ is a morphism $\varphi: g\cdot x \rightarrow y$. We define $\Phi(g,\varphi)$ as the $2$-morphism in the following commutative diagram

\begin{center}
\begin{tikzcd} [column sep=large, row sep=large]
G\times T \arrow{rr}{id_{G\times T}}
\arrow[rd, ""{name=U}]{}[swap]{\mu \circ (id\times g\cdot x)} 
&& G\times T \arrow{dl}{\mu \circ (id\times y)}[swap,""{name=V}]{} 
\arrow[Rightarrow, from=U, to=V, "id_G \times \varphi"]\\
&\mathcal{M}
\end{tikzcd}
\end{center}
\noindent where $\varphi$  is the morphism described by the commutative diagram

\begin{center}
\begin{tikzcd}
T \arrow{rr}{id_T}
\arrow[rd, ""{name=U}]{}[swap]{g\cdot x} 
&& T \arrow{dl}{y}[swap,""{name=V}]{} 
\arrow[Rightarrow, from=U, to=V, "\varphi"]\\
&\mathcal{M}
\end{tikzcd}
\end{center}

\noindent and $\Phi$ is a fully faithful morphism. To see this, we consider $T\in \Diff$ and
\[ \text{Hom}_{\mathcal{P}(T)} (x,y)\xrightarrow{\Phi_{x,y}}\text{Hom}_{\MM/G(T)}(\Phi(x), \Phi(y)). \]  
Then a morphism in $\text{Hom}_{\MM/G(T)}(\Phi(x), \Phi(y))$ is given by $h\times id_T: G\times T \rightarrow G\times T$, where $g_1\cdot g x \cong h(g_1)\cdot y  $, that means $h^{-1}(g_1)\cdot (g_1\cdot g x) \cong y$, but as $y\cong g\cdot x$ we have that $h^{-1}(g_1)\cdot g_1 \cong e$, the identity element of $G$. Therefore $h = id_G$, $\Phi_{x,y}$ is a bijection and $\Phi$ is fully faithful. We observe that this morphism is locally essentially surjective since its image is given by the trivial bundles.
This morphism therefore extends to an isomorphism  of stacks via stackification
$\Phi':(\mathcal{M}/G)^* \rightarrow \mathcal{M} /G$
as desired.\qedhere

\end{proof}

\begin{rem}
We observe that $\MM(T)$ can be considered as a subcategory of $\MM/G(T)$, where to each element $x\in \MM(T)$, by the $2$-Yoneda Lemma, we can associate the morphism $T\xrightarrow{x}\MM$ and assign the element in $\MM/G(T)$ given by the diagram 

\begin{center}
 \begin{tikzcd}[column sep=large]
G\times T \arrow{d}{p_2} \arrow{r}{\mu \circ (id \times x)} & \mathcal{M} \\
 T 
\end{tikzcd}
\end{center}
which is the trivial principal bundle over $T$.
\end{rem}

\begin{exm}
For any given stack $\mathcal{N}$, we can associate a $G$-stack  given by $(\mathcal{N}, pr_2, id, id)$, where $pr_2:G\times \mathcal{N} \rightarrow \mathcal{N}$ is the projection to the second component. We have the following $2$-commutative diagrams

\begin{center}
 \begin{tikzcd}[column sep=large]
G\times G\times \mathcal{N} (T) \arrow[r, ""{below, name=W, inner sep=1pt}, "m\times id_{\mathcal{N}(T)}"{above}]{} \arrow[d,""{right,name=X,inner sep=1pt},"id_G \times pr_2"{left}]
& G\times \mathcal{N} (T) \arrow{d}{pr_2} \\
G\times \mathcal{N} (T) \arrow{r}{pr_2} \arrow[ur,Rightarrow,"\alpha=id"]
& \mathcal{N} (T)
\end{tikzcd}
\end{center}
and

\begin{center}
\begin{tikzcd}[column sep=large]

G\times \mathcal{N} (T)
\arrow[r,"pr_2"{above}]
&\mathcal{N} (T) \\
\mathcal{N}(T)\arrow[u, "1 \times id_{\mathcal{N}(T)}"{left}]{} \arrow[ur," id_{\mathcal{N}(T)}"{swap}]{}  &\arrow[ul, Rightarrow,shorten <= 5em, shorten >= 0.02em, pos=0.85, "\beta=id"]
\end{tikzcd}
\end{center}
with
\begin{enumerate}
	\item $(g\cdot \alpha_{h,k}^x) \alpha_{g,hk}^x = \alpha_{g,h}^{k\cdot x}\alpha_{gh,k}^x$, for all $g,h,k \in G$ and $x\in \mathcal{N}(T)$.
	
	\item $(g\cdot \beta^x )\alpha_{g,e}^x = 1_{g\cdot x} = \beta^{g\cdot x} \alpha^{x}_{e,g} $ for every $g\in G$, $x\in \mathcal{N}(T)$ and with $e$ the identity in $G$,
\end{enumerate}
because $\alpha$ and $\beta$ are identities here. Therefore $(\mathcal{N}, pr_2, id, id)$ is a $G$-stack.
\end{exm}

\noindent In general, we can consider the $2$-functor $\iota:\St \rightarrow G-\St$ such that 

\begin{enumerate}
	\item for $\mathcal{N}\in \St$, we have $\iota(\mathcal{N}) = (\mathcal{N}, pr_2, id, id)$.
	\item For an $1$-morphism of stacks $\MM\xrightarrow{F}\mathcal{N}$, we have $\iota(\MM)\xrightarrow{\iota(F)}\iota(\mathcal{N})$ where $\iota(F)=(F, id)$.
	\item For a $2$-morphism of stacks $F\xrightarrow{\phi} F' $, we have $(F, id)\xrightarrow{\phi} (F',id) $.
\end{enumerate}

This is a $2$-functor since the identities provided by $\iota$ preserve all identities and all compositions. So we can verify that the notion of a quotient stack $\MM/G$ also coincides with the definition given by Romagny \cite[2.3]{romagny2003group} in the algebro-geometric context. 

\begin{prop}
The stack $\mathcal{M}/ G$ 2-represents the $2$-functor $\St\rightarrow \mathbf{Cat}$ defined by
\[ F(\mathcal{N})= Hom_{G-\St}(\mathcal{M},\iota(\mathcal{N})) \]
\end{prop}

\begin{proof}

Let $f\in Hom_{G-\St}(\mathcal{M},\iota(\mathcal{N}))$ be a morphism of $G$-stacks. If we consider the prestack $\mathcal{P}_{\mathcal{N}}$ associated to $\mathcal{N}$, we can see that any element in $\mathcal{P}_{\mathcal{N}}(T)$ is given by

\begin{center}
 \begin{tikzcd}[column sep=large]
G\times T \arrow{d}{p_2} \arrow{r}{pr_2 \circ (id \times x)} & \mathcal{M} \\
 T 
\end{tikzcd}
\end{center}
As the action is given by $pr_2$ we get that the elements in $\mathcal{P}_{\mathcal{N}}(T)$ are in bijective correspondence with the elements in $\mathcal{N}(T)$. Therefore if we use stackification we get that $\mathcal{N}/G \cong \mathcal{N}$. Hence we get an element in $Hom(\mathcal{M}/ G, \mathcal{N})$ and we have that $Hom_{G-\St}(\mathcal{M},\iota(\mathcal{N})) \cong Hom(\mathcal{M}/ G, \mathcal{N})$ by stackification.
\end{proof}

\begin{prop}
The projection morphism $\mathcal{M}\xrightarrow{q} \mathcal{M} / G$ has local sections.
\end{prop}

\begin{proof}

First we need to check how the morphism $q$ is defined. Let $V$ be a smooth manifold, then we have

\[ \mathcal{M}(V)\xrightarrow{q} \mathcal{M}/G (V) \]
\[ V\xrightarrow{f} \mathcal{M} \mapsto (G\times V \xrightarrow{pr_2} V , G\times V \xrightarrow{\mu \circ (id_G \times f)} \mathcal{M}  ) \]

\noindent We need to check that in the following diagram 

\begin{center}

\begin{tikzcd}

 V\times_{\mathcal{M}/G}\mathcal{M}  \arrow{r}\arrow{d} & \mathcal{M} \arrow[d, "q"] \\
 V\arrow[r, "p"] & \mathcal{M}/G
\end{tikzcd}

\end{center}
there exist local sections on $V\times_{\mathcal{M}/G}\mathcal{M}\rightarrow V$ that makes this diagram commute. For this, we consider a covering $\{ U_i \xrightarrow{i} V \}$ such that the $U_i$ are local trivialisation of $E\rightarrow V$. Then we get a diagram

\begin{center}

\begin{tikzcd} 

 G\times U_i  \arrow[r,"n"]\arrow{d} & E \arrow[d, "p"] \arrow[r, "h"] & \mathcal{M} \\
 U_i \arrow[r, "i"] & T &

\end{tikzcd}

\end{center}
If we consider the section $s_i:U_i \rightarrow G\times U_i$, then the section for $q$ is given by $s = h \circ n \circ s_i$ in $\MM(U_i)$. \qedhere

\end{proof}

Since we are working with $G$-stacks, we need to ask for an extra condition to obtain a differentiable $G$-stack, namely having a smooth manifold with a smooth action of $G$ and a morphism that preserves the action as follows: 

\begin{dfn} \label{G-DiffSt}

Let $G$ be a Lie group. A $G$-stack $\mathcal{M}$ is called a \textit{differentiable $G$-stack} if there is a smooth manifold $X$ with a smooth action $\sigma:G\times X \rightarrow X$ and a $1$-morphism of $G$-stacks $p:X\rightarrow \mathcal{M}$ such that:
\begin{enumerate}
	\item $p$ is representable.
	\item $p$ is a submersion.
\end{enumerate}
The morphism $p: X\rightarrow \mathcal{M}$ is then called a $G$-\textit{atlas}  \textit{for $\mathcal{M}$}.

\end{dfn}

\begin{rem}
We obtain a $2$-category of differentiable $G$-stacks denoted by $G$-$\Diff\St$.
\end{rem}

\begin{prop}\label{ActionInducedSimpManifold}
Let $\MM$ be a differentiable $G$-stack with $G$-atlas given by $X\xrightarrow{p}\MM$. If $\sigma$ is the smooth action of $G$ on $X$, this action induces a simplicial smooth action $\sigma_{\bullet}$ on the nerve of the associated Lie groupoid $(X\times_\MM X \rightrightarrows X)$.
\end{prop}
\begin{proof}
We consider the $(n+1)$-fold fibred product $X_n = X\times_\MM \ldots \times_\MM X$ and define 
$$\sigma_{n,T}: G\times X_n(T) \rightarrow X_n(T)$$
such that 
$$\sigma_{n,T}(g, (x_1, x_2, \ldots, x_{n+1} ; p(x_1)\Rightarrow\ldots \Rightarrow p(x_{n+1})))$$
$$= (g\cdot x_1,g\cdot x_2, \ldots, g\cdot x_{n+1} ; p(g\cdot x_1)\Rightarrow\ldots \Rightarrow p(g\cdot x_{n+1}))$$
This morphism is well-defined since we have $g\cdot p(x_1)\Rightarrow\ldots \Rightarrow g\cdot p(x_{n+1})$ and $g\cdot p(z) \cong p(g\cdot z)$ for any $z\in X(T)$, because $p$ is a $1$-morphism of $G$-stacks.
We observe that we obtain therefore the following diagrams:

\begin{center}
 \begin{tikzcd}[column sep=large]
G\times G\times X_n (T) \arrow[r, ""{below, name=W, inner sep=1pt}, "m\times id_{X_n(T)}"{above}]{} \arrow[d,""{right,name=X,inner sep=1pt},"id_G \times\sigma_{n,T}"{left}]
& G\times X_n (T) \arrow{d}{\sigma_{n,T}} \\
G\times X_n (T) \arrow{r}{\sigma_{n,T}} \arrow[ur,Rightarrow,"\alpha"]
& X_n (T)
\end{tikzcd}
\end{center}
and

\begin{center}
\begin{tikzcd}[column sep=large]

G\times X_n (T)
\arrow[r,"\sigma_{n,T}"{above}]
& X_n (T) \\
X_n(T)\arrow[u, "e \times id_{X_n(T)}"{left}]{} \arrow[ur," id_{X_n(T)}"{swap}]{}  &\arrow[ul, Rightarrow,shorten <= 5em, shorten >= 0.02em, pos=0.85, "\beta"]
\end{tikzcd}
\end{center}
where $\alpha$ and $\beta$ are identities. By construction, the collection of all these morphisms gives a simplicial map and it is a smooth simplicial map since it makes the following diagram 
\begin{center}
 \begin{tikzcd}[column sep=large]
 G\times X_n (T) \arrow[r, ""{below, name=W, inner sep=1pt}, "\sigma_{n,T}"{above}]{} \arrow[d,""{right,name=X,inner sep=1pt},""{left}]
&  X_n (T) \arrow{d} \\
G\times X (T) \arrow{r}{\sigma} 
& X (T)
\end{tikzcd}
\end{center}
commute, where the vertical maps are given by the different compositions of face maps of the simplicial manifold $X_\bullet$.
\end{proof}

We now have the following

\begin{prop} \label{AtlasQuotientStack}
Let $\mathcal{M}$ be a differentiable $G$-stack and let $X\xrightarrow{p}\mathcal{M}$ be a $G$-atlas. Then
$\mathcal{M}/G$ is a differentiable stack and the composition $X\rightarrow \mathcal{M}\xrightarrow{q}  \mathcal{M} / G$ is an atlas. \end{prop}

\begin{proof}
We observe that the morphisms $p$ and  $q$ have local sections, so it remains to check that $q\circ p$ is representable.
We consider the coverings $\{ U_i \rightarrow T \}$ and $\{ U_{ij}\rightarrow U_i \}$ such that the first one is a local section for $q$ and the second one a local section for $p$. Hence we get the following commutative diagram

\begin{center}

\begin{tikzcd}

& & U_{ij}\times_\mathcal{M} X \arrow{r} \arrow{d} & X \arrow{d}{p}\\
& & U_i\times_{\mathcal{M}/G}\mathcal{M}  \arrow{r}\arrow{d} & \mathcal{M} \arrow[d, "q"] \\
U_{ij} \arrow{rruu} \arrow{r}& U_i \arrow{ru} \arrow{r} & T\arrow{r} & \mathcal{M}/G
\end{tikzcd}

\end{center}
Thus if we glue every $U_{ij}\times_\mathcal{M} X$ using this local section, we get a smooth manifold and as the diagram commutes we see that $T\times_{\mathcal{M}/G} X$ is also a smooth manifold and from this we can conclude that the quotient stack $\mathcal{M}/G$ is in fact a differentiable stack. \qedhere
\end{proof}

As the morphism $q: \mathcal{M}\rightarrow  \mathcal{M}/ G$ is in fact a principal $G$-bundle (see \cite[4.2]{ginot2012group}), there exist an associated classifying morphism 
$u:\mathcal{M}/ G\rightarrow \mathcal{B}G$ to the classifying stack $\mathcal{B}G$ and we obtain the following $2$-cartesian square
\begin{center}
 \begin{tikzcd}
\mathcal{M} \arrow{r} \arrow{d}{q}
& *  \arrow{d} & \\
\mathcal{M}/G  \arrow{r}{u}
& \mathcal{B}G &
\end{tikzcd}
\end{center}

The canonical morphism of differentiable stacks $q: \mathcal{M}\rightarrow  \mathcal{M}/ G$ is actually the universal principal $G$-bundle over $\mathcal{M}/ G$. 

Let us mention finally that there is also a dual notion of a fixed point stack $\MM^G$, providing stacky fixed points for general Lie group actions on differentiable stacks and which is related to a homotopy fixed point space $\Hom_G(EG, \MM)$. In the context of group scheme actions these were also studied by Romagny \cite{romagny2003group} for the construction of certain moduli stacks with symmetries. We aim to explore these notions systematically in the context of differentiable $G$-stacks in a follow-up article. 

\section{Cohomology of differentiable stacks}

\noindent We will now discuss several cohomology theories for differentiable stacks, namely de Rham cohomology, sheaf cohomology and hypercohomology. As general references for cohomology of stacks, in particular de Rham and sheaf cohomology, we refer to \cite{behrend2004cohomology}, \cite{behrend2011differentiable} and \cite{heinloth2005notes}.

\subsection{De Rham cohomology}\label{DRCohSubsection}

Let $\MM$ be a differentiable stack with an atlas $X\rightarrow \MM$. We consider the nerve $X_\bullet$ of its associated Lie groupoid $(X\times_\MM X \rightrightarrows X)$ and the de Rham complex
\[ \left(\Omega^{\ast}_{dR}(X_\bullet), D= d_{dR}+ \partial \right), \] 
where $\partial$ is the differential operator defined via pullbacks of face maps in the simplicial structure of the nerve of $(X\times_\MM X \rightrightarrows X)$ and $d_{dR}$ is the exterior derivative of the de Rham complex for smooth manifolds.
Where we are considering
$$\Omega^{n}_{dR}(X_\bullet)=\displaystyle\bigoplus_{p+q=n} \Omega_{dR}^q(X_p)$$
with $\Omega_{dR}^q(X_p)$ denoting the sheaf of differential $q$-forms on the smooth manifold $X_p$.

\begin{dfn}
The cohomology given by the complex $(\Omega^{\ast}_{dR}(X_\bullet), D )$ is called the \textit{de Rham cohomology of $X_\bullet$}, and is denoted by $H^\ast_{dR}(X_\bullet)$.
\end{dfn}

This cohomology is invariant under Morita equivalence of Lie groupoids (see \cite{behrend2004cohomology}). Thus, it gives a well-defined algebraic invariant for differentiable stacks and we can define using an atlas

\begin{dfn}
Let $\MM$ be a differentiable stack with atlas $X\rightarrow \MM$. The \textit{de Rham cohomology} of $\MM$ is defined as
\[ H^\ast_{dR}(\mathcal{M}):= H^\ast_{dR}(X_\bullet). \] 
\end{dfn}

It can be shown that the de Rham cohomology of $X_\bullet$ is isomorphic to the singular cohomology of its fat geometric realisation, where the \textit{fat geometric realisation} is given as the quotient space
$$ \Vert X_\bullet \Vert=  \Vert p \mapsto X_p \Vert  = \bigcup_{p\in \mathbb{N}} \Delta^p \times X_p / \sim $$
modulo the identifications $(\partial^i t, x)\sim (t, \partial_i x)$ for any $x\in X_p$, $t\in \Delta^{p-1}, i, j=0, \ldots, n$ and $p$ (compare \cite[2.8]{dupont1976simplicial}). 

There is also a de Rham theorem for differentiable stacks (see \cite{behrend2004cohomology}), namely for a given differentiable stack $\MM$, we have an isomorphism between its de Rham cohomology and its singular cohomology with real coefficients i.e. $$H^\ast_{dR}(\mathcal{M})\cong H^\ast_{sing}(\mathcal{M}, \R).$$

Furthermore we can define a homotopy type for a differentiable stack

\begin{dfn}\label{HomotopyTypeStacks}
Let $\MM$ be a differentiable stack with atlas $X\rightarrow \MM$. 
The {\it homotopy type} of the differentiable stack $\MM$ is given by the homotopy type of the fat geometric realisation $|| X_\bullet ||$ of the nerve $X_\bullet$ of the associated Lie groupoid $(X\times_\MM X \rightrightarrows X)$.
\end{dfn}

The homotopy type of a differentiable stack does not depend on the choice of atlas and therefore is again well-defined (compare \cite {noohi2012homotopy}).
 
\subsection{Sheaf cohomology and hypercohomology}

We will now consider sheaf cohomology and hypercohomology of differentiable stacks, which again rely on the simplicial constructions. We refer to \cite{heinloth2005notes}, \cite{felder2008gerbe} and \cite{behrend2011differentiable} for sheaves and complexes of sheaves on differentiable stacks and their cohomologies. 

\begin{dfn} \label{Sheaf}

A \textit{sheaf} $\mathfrak{F}$ \textit{(of abelian groups)} \textit{on a differentiable stack} $\mathcal{M}$ is a collection of sheaves $\mathfrak{F}_{X\rightarrow \mathcal{M}}$ (of abelian groups) for any morphism $X\rightarrow \mathcal{M}$, where $X$ is a smooth manifold such that for every triangle 

\begin{center}
\begin{tikzcd}
X \arrow{rr}{f}
\arrow[rd, ""{name=U}]{}[swap]{h} 
&& Y \arrow{dl}{g}[swap,""{name=V}]{} 
\arrow[Rightarrow, from=U, to=V, "\phi"]\\
&\mathcal{M}
\end{tikzcd}
\end{center}
there is a morphism of sheaves $\Phi_{\phi, f} : f^*\mathfrak{F}_{Y\rightarrow \mathcal{M}} \rightarrow\mathfrak{F}_{X\rightarrow \mathcal{M}}$ such that for

\begin{center}
\begin{tikzcd}
X \arrow{r}{f}
\arrow[rd, ""{name=U}]{}[swap]{}
&Y \arrow[d, ""{name=W}]{}[swap,""{name=V}]{}
\arrow{r}{g}
&Z \arrow{ld}{}[swap,""{name=X} ]{}
\arrow[Rightarrow, from=U, to=V, "\phi"]
\arrow[Rightarrow, from=W, to=X, "\psi"]\\
&\mathcal{M}
\end{tikzcd}
\end{center}
we have $\Phi_{\phi, f} \circ f^* \Phi_{\psi, g} = \Phi_{\phi \circ f^*\psi, f\circ g}$. The sheaf $\mathfrak{F}$ is called \textit{cartesian} if the $\Phi_{\phi, f}$ are also isomorphisms.
\end{dfn}

Let $\mathfrak{F}$  be a sheaf of abelian groups on a differentiable stack $\mathcal{M}$. We consider an injective resolution $0\rightarrow \mathfrak{F} \rightarrow K^{q}$ and the double complex of global sections 
\[ N^{\bullet,\bullet} = \Gamma(X_{\bullet}, K^{\bullet})  \]
with differentials given by the induced simplicial differential $\partial$ and the resolution differential $d_K$. This cohomology is the \textit{sheaf cohomology} of $\MM$ and will be denoted by $H^*(\mathcal{M}, \mathfrak{F})$.

The same construction as above can be done more generally also for a complex $\mathfrak{F}^\bullet$ of sheaves of abelian groups on the differentiable stack $\MM$ with atlas $X\rightarrow \MM$ (see \cite{behrend2004cohomology}, \cite{heinloth2005notes} or    
 \cite{felder2008gerbe}). The result is the \textit{hypercohomology of} $\MM$ \textit{with coefficients} $\mathfrak{F}^\bullet$ and will be denoted by
\[ H^\ast(\MM,\mathfrak{F}_0 \rightarrow \mathfrak{F}_1 \rightarrow \cdots \rightarrow \mathfrak{F}_m ) = H^\ast(\MM,\mathfrak{F}^\bullet) = H^\ast(X_\bullet,\mathfrak{F}^\bullet) . \]

\begin{rem}\label{ComplexSimplicialCohomology}
If we have a complex $\mathfrak{F}^\bullet$ of cartesian sheaves, this cohomology is isomorphic to the cohomology of the homotopy type of the fat geometric realisation $||X_\bullet||$ (compare with \cite[3.4.27]{deligne1974theorie}). 
\end{rem}

\section{Models for equivariant cohomology of differentiable stacks} \label{CohQSt}

\noindent In this section we will now describe the Borel, Cartan and Getzler models for equivariant cohomology of differentiable stacks with Lie group actions and study some of their main properties. 

\subsection{The homotopy type of the differentiable stack \texorpdfstring{$\MM/G$}{M/G} and the Borel model for differentiable \texorpdfstring{$G$}{G}-stacks}
Let $G$ be a Lie group and $\MM$ a differentiable $G$-stack with a $G$-atlas $X\xrightarrow{p} M$. We denote the action on $\MM$ by $G$ with $\mu:G\times \MM \rightarrow \MM$ and the action on $X$ by $G$ with $\sigma:G\times X \rightarrow X$. Then by proposition \ref{AtlasQuotientStack}, we obtain an atlas for the quotient stack $X\xrightarrow{p}\MM\xrightarrow{q}\MM/G$. We recall that the homotopy type for the quotient stack $\MM/G$ is given by the fat geometric realisation of the nerve of the associated Lie groupoid $(X\times_{\MM/G}X\rightrightarrows X)$ (compare also with \cite{behrend2004cohomology}).

For the rest of this section we will devote our efforts to determine the homotopy type of the differentiable stack $\MM/G$. We will utilise the following result by Ginot-Noohi (see \cite[4.1]{ginot2012group}) describing the $2$-fiber product of the morphism $q$.

\begin{prop}
The following diagram is a $2$-commutative diagram 
\begin{center}
 \begin{tikzcd}
G\times \MM \arrow{r}{\mu} \arrow{d}{pr_2}
& \MM \arrow{d}{q} \\
\MM \arrow{r}{q} 
& \MM/G
\end{tikzcd}
\end{center}
and the functor $(pr_2, \mu): G\times \MM \rightarrow \MM \times_{\MM/G} \MM$ is an equivalence of stacks.

\end{prop} 

Therefore, we can consider the following $2$-commutative diagram, which is a modification of the one above, using the fact that there is also an induced action of the Lie group $G$ on the atlas $X$ of the stack $\MM$:

\begin{center}

\begin{tikzcd}
E \arrow[r] \arrow[d, "\mu_1 "]
& G\times X \arrow[r, "\sigma"] \arrow[d," id_G \times p "] & X \arrow[d,"p"]\\
 G\times X \arrow[r, "id_G\times p" ] \arrow[d, "pr_2"]
& G\times \mathcal{M}  \arrow[r, "\mu "]\arrow[d, "pr_2"] & \mathcal{M} \arrow[d, "q"] \\
X\arrow[r, "p"] & \mathcal{M}\arrow[r, "q"] & \mathcal{M}/G
\end{tikzcd}

\end{center} 

From this we can now conclude the following crucial property

\begin{prop}
There is an equivalence of stacks given by
$$ X\times_{\MM/G}X \cong (G\times X)\times_\MM X \cong G\times (X\times_\MM X) $$
\end{prop}

\begin{proof}
We obtain the first equivalence thanks to the diagram above, so we have 
$$X\times_{\MM/G}X \cong (G\times X)\times_\MM X.$$
To show the equivalence $(G\times X)\times_\MM X \cong G\times (X\times_\MM X) $, we consider any $T$ in $\Diff$ and we observe that any element in $(G\times X)\times_\MM X(T)$ is of the form $(g\times x, y ; g\cdot p(x)\Rightarrow p(y))$, where $x, y\in X(T)$ and $g\in G$. In the same way, an element in $G\times (X\times_\MM X)(T)$ has the form $(g, (x,y; p(x)\Rightarrow p(y)))$. Now we define the morphism
$$ \eta_T:(G\times X)\times_\MM X(T)\rightarrow G\times (X\times_\MM X)(T) $$
with $\eta_T (g\times x, y ; g\cdot p(x)\Rightarrow p(y)) = (g, (x,g^{-1}\cdot y; p(x)\Rightarrow p(g^{-1}\cdot y))) $, where $g^{-1}\cdot y = \sigma(g^{-1},y)$ and the morphism
$$ \xi_T: G\times (X\times_\MM X)(T) \rightarrow (G\times X)\times_\MM X(T)$$
with $\xi_T (g, (x,y; p(x)\Rightarrow p(y)))=(g \times x, g\cdot y; g\cdot p(x)\Rightarrow p(g\cdot y))$.
Finally, we get that $\eta_T \circ \xi_T = id_{G\times (X\times_\MM X)(T)}$ and $\xi_T \circ \eta_T = id_{(G\times X)\times_\MM X(T)} $, as we desired.\qedhere
\end{proof}

Iterating the above equivalence of stacks, we now get for the $(n+1)$-fold fibred product 
$$X\times_{\MM/G}X\times_{\MM/G}\ldots \times_{\MM/G} X \cong G^{n}\times X_{n},$$
where $G^{n}$ is the $n$-fold cartesian product of $G$ and $X_{n}= X\times_{\MM}X\times_{\MM}\ldots \times_{\MM} X$ is the $(n+1)$-fold fibred product.

If we consider again any manifold $T$ in $\Diff$ and the face maps in the nerve associated to $X \xrightarrow{q\circ p} \mathcal{M}/G$,  we obtain the face maps for the simplicial smooth manifold $\{G^n\times X_{n}\}_{n\geq 0}$
$$ \partial_i:G^{n+1}\times X_{n+1}(T) \rightarrow G^{n}\times X_{n}(T) $$
such that 
$$ \partial_i (g_{1},g_2,\ldots , g_{n+1}, (x_1, x_{2},\ldots , x_{n+2} ; p(x_{1})\Rightarrow\ldots\Rightarrow p(x_{n+2}) ))$$
is equal to
$$ (g_{2},g_3,\ldots , g_{n+1}, \pi_1 (x_1, x_{2},\ldots , x_{n+2} ; p(x_{1})\Rightarrow\ldots\Rightarrow p(x_{n+2}) ) ) \text{ if $i=0$, }$$
$$ (g_{1},\ldots , g_i\cdot g_{i+1},\ldots , g_{n+1}, \pi_i(x_1, x_{2},\ldots , x_{n+2} ; p(x_{1})\Rightarrow\ldots\Rightarrow p(x_{n+2}) )) \text{ if $0 < i < n$, }$$
$$ (g_{1},g_2,\ldots , g_{n}, g_{n+1}\cdot \pi_{n+2} (x_1, x_{2},\ldots , x_{n+2} ; p(x_{1})\Rightarrow\ldots\Rightarrow p(x_{n+2}) ) ) \text{ if $i=n$, }$$
and where $\pi_i:X_{n+1}\rightarrow X_{n} $ is the $i$-th face map of the nerve of the simplicial smooth manifold $X_\bullet$ and $g_{n+1}\cdot \pi_{n+2}$ is the action induced by $\sigma$ on $X_{n}$. Now we can state a fundamental theorem

\begin{teo}\label{TeoCohomologyQuotientStack}
Let $G$ be a Lie group and $\MM$ a differentiable $G$-stack with $G$-atlas $X\rightarrow \MM$ then 
\[ H^\ast(\MM/G, \R) \cong H(EG\times_G \Vert X_\bullet\Vert, \R), \]
where $||X_\bullet||$ is the fat geometric realisation of the associated simplicial smooth manifold $X_\bullet$.
\end{teo}

\begin{proof}
We consider the bisimplicial smooth manifold $Z_{\bullet,\bullet}$ given by $Z_{p,n}= G^p \times X_n $, where $G^p$ is the $p$-fold cartesian product of $G$ and $X_{n}= X\times_{\MM}X\times_{\MM}\ldots \times_{\MM} X$ is the $(n+1)$-fold fibred product of the $G$-atlas for $\MM$. We have $X_0=X$. The vertical face maps are given by the face maps in the simplicial manifold $X_\bullet$ and the horizontal face maps are given by 
\[\partial_0^H(g_1,\ldots,g_p, z)= (g_2,\ldots, g_p,z) \]
	\[\partial_i^H(g_1,\ldots,g_p, z)= (g_1,\ldots g_{i-1},g_ig_{i+1},g_{i+2},\ldots, g_p,z) \text{ for } 1\leq i < p \]
	\[ \partial_p^H(g_1,\ldots,g_p, z) = (g_1,\ldots, g_{p-1},g_p\cdot z)\]
where $z\in X_n$. We observe that the diagonal simplicial smooth manifold $dZ_{\bullet}$ associated to $Z_{\bullet,\bullet}$ given by $\{Z_{n,n}\}_{n\geq 0}$ is the simplicial smooth manifold given by $\{G^n\times X_{n}\}_{n\geq 0}$ that we discussed above. 
As the fat geometric realisation $\Vert Z_{\bullet, \bullet} \Vert$ of a bisimplicial smooth manifold $Z_{\bullet, \bullet}$ can be computed using the fat geometric realisation of the simplicial smooth manifold in the second index and then the one in the first index $\Vert p\mapsto \Vert n \mapsto Z_{p,n} \Vert \Vert$, we get that 
\[ \Vert Z_{\bullet, \bullet} \Vert = \Vert G^\bullet \times X_\bullet\Vert \cong \Vert p\mapsto \Vert n \mapsto G^p \times X_n \Vert \Vert \cong EG\times_G\Vert X_\bullet\Vert. \]
Alternatively, the fat geometric realisation $\Vert G^\bullet \times X_\bullet\Vert$ can also be computed via the fat geometric realisation of the diagonal simplicial smooth manifold $dZ_{\bullet}$ of the bisimplicial smooth manifold $Z_{\bullet, \bullet}$ (see also \cite{felisattineu2007sim, ebertrandall2019semi}), that is, 
\[ EG\times_G\Vert X_\bullet\Vert \cong \Vert n\mapsto Z_{n,n} \Vert \cong \Vert n\mapsto G^n \times X_{n} \Vert  \]
which implies the desired result.
\end{proof}

\begin{exm}
If the differentiable stack $\mathcal{M}$ is just a smooth manifold $X$, we see that the definition of equivariant cohomology for stacks coincides with the usual equivariant cohomology for smooth manifolds. 
Namely, since 
\[ (G\times X) \times_X X \cong G \times X\]
with the maps of the associated Lie groupoid $(G\times X\rightrightarrows X)$ given by the action map $\mu:G\times X \rightarrow X$ and the projection map $pr_2:G\times X \rightarrow X$ respectively, we observe that this Lie groupoid coincides with the transformation groupoid and the fat geometric realisation gives $EG\times_G X$. Therefore the de Rham cohomology of $[X/G]$ is simply given as 
$$H^\ast_{dR}([X/G])\cong H^\ast(EG\times_G X, \R),$$ 
which is the classical Borel model for equivariant cohomology.
\end{exm}

Now in the general situation we define in analogy with the classical case of $G$-manifolds
\begin{dfn}
Let $G$ be a Lie group and $\MM$ a differentiable $G$-stack with a $G$-atlas $X\xrightarrow{p} M$. The \textit{equivariant cohomology} $H_G^\ast(\MM, R)$ of $\MM$ is given by
$$ H_G^\ast(\MM, R) = H^\ast (\MM/G, R),$$
where $R$ is any commutative ring $R$ with unit.
\end{dfn}

\begin{rem}
By \cite[3.4.27]{deligne1974theorie}, this definition of equivariant cohomology in fact makes sense for any cartesian sheaf or complex of cartesian sheaves given on the quotient stack $\MM/G$. 
\end{rem}

\subsection{The Cartan model for differentiable \texorpdfstring{$G$}{G}-stacks}\label{CartanModelSubsection}
In this subsection, we use the Cartan model for simplicial smooth manifold as described in \cite{meinrenken2005witten}, \cite{suzuki2015simplicial} to obtain another description of the equivariant cohomology of a differentiable $G$-stack. We will assume in this section that $G$ is a {\it compact} Lie group.

Let $\MM$ be a differentiable $G$-stack with $G$-atlas $X\rightarrow \MM$ for a compact Lie group $G$. We denote the action on the atlas $X$ by $\sigma$ and the action on the stack $\MM$ by $\mu$. Then we can consider the simplicial smooth action $\sigma_\bullet$ induced by $\sigma$ on $X_\bullet$ as in Proposition \ref{ActionInducedSimpManifold}.

Let us now consider the associated complex $(C^{\bullet}, D - \iota)$ of simplicial equivariant forms given by
\[ C^{2p+m} = \bigoplus_{q+r=m} \left( S^{p}(\mathfrak{g}^\vee)\otimes \Omega_{dR}^q(X_r)^G \right), \]
where $\mathfrak{g}$ denotes the Lie algebra of $G$, with $D$ the differential defined for the de Rham complex as in Subsection \ref{DRCohSubsection} and $\iota$ is the interior multiplication by the fundamental vector field (compare \cite{kubel2015equivariant}, \cite{suzuki2015simplicial}). We consider its cohomology and denote it by $H_G^\ast(X_\bullet)$. We observe that
\begin{teo}\cite[4.2]{suzuki2016equivariant}, \cite[4.1]{suzuki2015simplicial}
The cohomology of the complex $C^\bullet$ is given by
\[ H_G^\ast (X_\bullet) \cong H^\ast(EG\times_G || X_\bullet ||, \mathbb{R}) \]
with $|| X_\bullet ||$ the fat geometric realisation of the simplicial smooth manifold $X_\bullet$.
\end{teo}

If we compare this result with theorem \ref{TeoCohomologyQuotientStack}, we see that this Cartan model and the cohomology of the quotient stack $\MM/G$ coincide, that is, for the equivariant cohomology of $\MM$ we have
\[ H^\ast_G(\MM, \R)=H^\ast(\MM/G, \R) \cong H_G^\ast (X_\bullet). \]

If we follow the general argument for differential graded algebras, as in \cite{guillemin2013supersymmetry}, we can deduce several properties concerning restrictions of the acting Lie group to a subgroup.

We recall the description of the double complex $(C^{\bullet, \bullet}, D, \iota)$ of invariant forms with
\[C^{p,q} = \left(S^{p}(\mathfrak{g}^\vee)\otimes \left(\bigoplus_{s+r=q-p}\Omega_{dR}^s(X_r)\right)\right)^G\]
with vertical operator $D$ and horizontal operator $\iota$. 

Filtering this double complex $C^{\bullet, \bullet}=\{C^{p, q}\}$ in the standard way gives rise to a spectral sequence converging to the graded associated of the equivariant cohomology
$H^\ast_G(\MM, \R)$ (compare \cite[Chap. 6]{guillemin2013supersymmetry}, \cite[Chap. 3]{botttu1982algtop}). We obtain the following identification:

\begin{teo}\label{E1SpectralsequenceEqCoh}
The $E_1$-term of the spectral sequence for the double complex $C^{\bullet, \bullet}$ of invariant forms is given as
$$ E_1^{p,q}=(S^{p}(\mathfrak{g}^\vee)\otimes H^{q-p}(X_\bullet, \mathbb{R}))^G.$$  
\end{teo}

\begin{proof}
The double complex $C^{\bullet, \bullet}=\{C^{p, q}\}$ with the boundary $ D$ sits inside the double complex $K^{\bullet, \bullet}=(S^{\ast}(\mathfrak{g}^\vee)\otimes (\bigoplus \Omega_{dR}^\ast (X_\bullet)))$ and the cohomology groups of  the components of $C^{\bullet, \bullet}$ are just the $G$-invariant components of the cohomology groups of the components of $K^{\bullet, \bullet}$ which are appropriately graded components of $S^*(\mathfrak{g}^\vee) \otimes H^*(X_\bullet, \mathbb{R})$.
\end{proof}

Some additional properties of the $E_1$-term can be obtained in the case when the Lie group $G$ is connected. 

\begin{prop}\label{CompoConexLie}
The connected component of the identity in $G$ acts trivially on the singular cohomology $H^\ast(X_\bullet, \mathbb{R})$.
\end{prop}

\begin{proof}
We consider the operator
\[ \iota D +  D \iota =  \iota (d_{dR}+(-1)^{q+1} \partial) +  (d_{dR}+(-1)^{q} \partial) \iota =  \iota d_{dR}+ d_{dR} \iota = L_\alpha. \]
Therefore the Lie derivative $ L_\alpha$ is chain homotopic to $0$ in $\Omega_{dR}^q(X_r)$. As for the Lie derivative the action is trivial in a connected component of the identity, we have the result.
\end{proof}

\begin{teo}
If $G$ is connected, then $E_1^{p,q}=S^p(\mathfrak{g}^\vee)^G \otimes H^{q-p}(X_\bullet, \mathbb{R})$
\end{teo}
\begin{proof}
By proposition \ref{CompoConexLie}, we see that $G$ acts trivially on $H^\ast(X_\bullet, \mathbb{R})$ and so from the description of the $E_1$-term of the spectral sequence in theorem \ref{E1SpectralsequenceEqCoh} we get the desired result.
\end{proof}

The above theorem helps us to get a relation between the equivariant cohomology for an action of a compact connected Lie group G and the action of a closed subgroup $K$ of $G$ in the following way

\begin{teo} \label{RestGr}
Let $G$ be a compact connected Lie group and $K$ a closed subgroup of $G$. Suppose that the restriction map 
\[S(\mathfrak{g}^\vee)^G \rightarrow S(\mathfrak{k}^\vee)^K \]
is an isomorphism. Then the induced map in equivariant cohomology is also an isomorphism 
\[ H_G(\MM, \mathbb{R})\xrightarrow{\cong} H_K(\MM, \mathbb{R}). \]

\end{teo} 
 \begin{proof}
 
 We have an injection of Lie algebras
\[\mathfrak{k}\rightarrow \mathfrak{g}\]
where $\mathfrak{k}$ is the Lie algebra of $K$ and $\mathfrak{g}$ is the Lie algebra of $G$.
We also have an injection between the dual spaces
\[\mathfrak{g}^\vee \rightarrow \mathfrak{k}^\vee, \]
which extends to the associated symmetric algebras
\[S(\mathfrak{g}^\vee) \rightarrow S(\mathfrak{k}^\vee) \]
and then to
\[ (S(\mathfrak{g}^\vee)\otimes \bigoplus \Omega_{dR}^\ast (X_\bullet ))^G \rightarrow (S(\mathfrak{k}^\vee)\otimes \bigoplus \Omega_{dR}^\ast (X_\bullet ))^K. \]
Therefore we get an induced homomorphism
\[ H_G(X_\bullet, \mathbb{R})\rightarrow H_K(X_\bullet, \mathbb{R}) \]
and also an induced homomorphism at each stage of the corresponding spectral sequences. Since $G$ acts trivially on $H^\ast(X_\bullet, \mathbb{R})$ and $K$ is a subgroup of $G$, the group $K$ acts trivially as well. So by Theorem \ref{E1SpectralsequenceEqCoh}, we get
a morphism of the $E_1$-terms
\[S(\mathfrak{g}^\vee)^G \otimes \bigoplus \Omega_{dR}^\ast (X_\bullet )\rightarrow S(\mathfrak{k}^\vee)^K\otimes \bigoplus \Omega_{dR}^\ast (X_\bullet ) \] 
and therefore an isomorphism
$$ H_G(X_\bullet, \mathbb{R})\xrightarrow{\cong} H_K(X_\bullet, \mathbb{R}).   $$
which gives the desired result.
 \end{proof}
 
We also obtain the following generalisation of the classical situation for equivariant cohomology of smooth manifolds 

\begin{teo}\label{TorusMaxSimplicial}
Let $G$ be a connected compact Lie group, $T$ a maximal torus of $G$ and $W$ the Weyl group. Futhermore, let $\MM$ be a differentiable $G$-stack. Then we have
\[ H_G^\ast(\MM,\mathbb{R})\cong H_T^\ast(\MM,\mathbb{R})^W. \]
\end{teo}
\begin{proof}
Let $T$ be a maximal torus of $G$ and let $K=N(T)$ be its normalizer. The Weyl group $W$ of $G$ is the quotient group $W=K/T$. It is a finite group and the Lie algebra of $K$ is the same as the Lie algebra of $T$. Since $T$ is abelian its action on $\mathfrak{t}^\vee$ and on $S(\mathfrak{t}^\vee)$ is trivial. So we get
\[ S(\mathfrak{k}^\vee)^K = S(\mathfrak{t}^\vee)^K = S(\mathfrak{t}^\vee)^W \]
According to the Chevalley restriction theorem the restriction morphism
\[S(\mathfrak{g}^\vee)^G\rightarrow S(\mathfrak{t}^\vee)^W \]
is an isomorphism (see for example \cite[2.1.5.1]{warner2012harmonic}). Therefore we can apply Theorem \ref{RestGr}.
Considering the inclusion $T\rightarrow K$ we get a morphism of double complexes
\[ C^\bullet_K(X_\bullet)\rightarrow C^\bullet_T(X_\bullet)^W \]
which induces a morphism
\[ H^\ast_K(X_\bullet, \R)\rightarrow H^\ast_T(X_\bullet, \R)^W \]
as well as a morphism at each stage of the spectral sequences. At the $E_1$-stage we obtain the identity morphism 
\[ S(\mathfrak{t}^\vee)^W \otimes H^\ast(X_\bullet,\mathbb{R}) \rightarrow S(\mathfrak{t}^\vee)^W \otimes H^\ast(X_\bullet,\mathbb{R}). \] 
Therefore we get that $H_K^\ast(X_\bullet,\mathbb{R})\cong H_T^\ast(X_\bullet,\mathbb{R})^W$ and so $H_G^\ast(\MM,\mathbb{R})\cong H_T^\ast(\MM,\mathbb{R})^W$, as a consequence of Theorem \ref{RestGr}. \qedhere
\end{proof}

\subsection{The Getzler model for differentiable \texorpdfstring{$G$}{G}-stacks}
Finally, we extend the Getzler model \cite{getzler1994equivariant} for equivariant cohomology of smooth manifold with arbitrary Lie group actions to differentiable stacks. Our description is closely related to the Getzler model as described by K\"ubel-Thom in \cite{kubel2015equivariant}. 

Let $\MM$ be a differentiable $G$-stack with $G$-atlas $X\rightarrow \MM$ for a general Lie group $G$ acting smoothly on $\MM$. The action on $\MM$ will always be denoted by $\mu$ and the one on $X$ by $\sigma$. We will also consider the induced simplicial smooth action $\sigma_\bullet$ on the associated simplicial smooth manifold $X_\bullet=\{X_n\}_{n\geq 0}$, i.e. the nerve of the associated Lie groupoid $(X\times_\MM X \rightrightarrows X)$, which we described in Proposition \ref{ActionInducedSimpManifold}.

Let us first consider the Getzler complex for the smooth manifolds $X_n$. We define the vector spaces $C^p(G,S^{\ast}(\mathfrak{g}^\vee)\otimes \Omega_{dR}^\ast(X_n))$ of smooth maps 
\[ G^p\rightarrow S^{\ast}(\mathfrak{g}^\vee)\otimes \Omega_{dR}^\ast(X_n) \]
from the $p$-fold cartesian product $G^p$ to $S^{\ast}(\mathfrak{g}^\vee)\otimes \Omega_{dR}^\ast(X_n)$, i.e. the polynomial maps from $\mathfrak{g}$ to the differential forms on $X_n$. The complex $C^{\bullet}$ is then given by
\[ \bigoplus_{p+2l+m=s}C^p(G,S^{l}(\mathfrak{g}^\vee)\otimes \Omega_{dR}^m(X_n)) \]
endowed with the differential $\d+\i+ (-1)^p (d+ \iota)$ as defined by  Getzler in \cite{getzler1994equivariant}, where  $d+\iota$ is the Cartan differential operator as in the Cartan model, the operator $\d$ is given by
\[ \d: C^p(G,S^{\ast}(\mathfrak{g}^\vee)\otimes \Omega_{dR}^\ast(X_n))\rightarrow C^{p+1}(G,S^{\ast}(\mathfrak{g}^\vee)\otimes \Omega_{dR}^\ast(X_n)) \]
such that
\[ (\d f)(g_0,\ldots,g_k\mid Y) := f(g_1,\ldots,g_k\mid Y) + \sum_{i=1}^k f(g_0,\ldots, g_{i-1}g_i,\ldots,g_k\mid Y) \]
\[ + (-1)^{k+1}f(g_0,\ldots,g_{k-1}\mid \text{Ad}(g_k^{-1})Y) \]
for $g_0, \ldots, g_k \in G$ and $Y\in \mathfrak{g}$, and the operator $\i$ is defined by
\[ \i: C^p(G,S^{l}(\mathfrak{g}^\vee)\otimes \Omega_{dR}^m(X_n))\rightarrow C^{p-1}(G,S^{l+1}(\mathfrak{g}^\vee)\otimes \Omega_{dR}^m(X_n)) \]
with
\[ (\i f)(g_1,\ldots,g_{p-1}\mid Y) := \sum_{i=0}^{p-1} \frac{d}{dt}\bigg\lvert_{t=0} f(g_1,\ldots, g_i, \text{exp}(t Y_i), g_{i+1},\ldots,g_{p-1}\mid Y) \]
where $Y_i=\text{Ad}(g_{i+1}\cdots g_{p-1})Y$. 

However we can also consider the simplicial differential $\partial_X$ for the associated simplicial smooth manifold $X_\bullet$. Because $\partial_X$ commutes with the operator $\d+\i+ (-1)^p (d+ \iota)$, we can therefore consider the complex $C^{\bullet}$ with
\[\bigoplus_{p+2l+m+n=s}C^p(G,S^{l}(\mathfrak{g}^\vee)\otimes \Omega_{dR}^m(X_n)) \]
and differential $ \d+\i+ (-1)^p (d+ \iota)+ (-1)^{p+2l+m}\partial_X $. The cohomology $H^*_G(X_\bullet)$ of this complex gives the Getzler 
model for differentiable $G$-stacks and it calculates the equivariant cohomology as the following result shows
\begin{teo}
Let $\MM$ be a differentiable $G$-stack with $G$-atlas $X\rightarrow \MM$ for a general Lie group $G$. Then there is an isomorphism
\[ H^\ast_G(\MM, \mathbb{R})\cong H^*_G(X_\bullet). \]
\end{teo}
\begin{proof}
To get the result we are going to use a spectral sequence argument. Let us consider the complex $\bigoplus_{p+2l+m+n=s}C^p(G,S^{l}(\mathfrak{g}^\vee)\otimes \Omega_{dR}^m(X_n))$  as the double complex $E_0^{' \bullet, \bullet}$ with
\[\left( E_0^{'s,n} = \left( \bigoplus_{p+2l+m=s, n }C^p(G,S^{l}(\mathfrak{g}^\vee)\otimes \Omega_{dR}^m(X_n)) \right), \d+\i+ (-1)^p (d+ \iota) , \partial_X \right) \]
In the same way, we consider the triple complex $(\Omega_{dR}^q(G^p\times X_n), d_{dR}, \partial_G,\partial_X)$ as the double complex $E_0^{\bullet, \bullet}$ with
\[ \left( E_0^{s,n} = \left( \bigoplus_{p+q=s, n } \Omega_{dR}^q(G^p\times X_n) \right), d_{dR}+(-1)^q\partial_G , \partial_X \right)\]
As Getzler showed (see \cite[2.2.3]{getzler1994equivariant}), we know that if we calculate both cohomologies first for the index $s$, we will get the same cohomology. This is due to the fact that the cohomology of $E_0^{s,n}$ is the cohomology of the associated Lie groupoid $(G\times X_n \rightrightarrows X_n)$ for the action of $G$ on the smooth manifold $X_n$, i.e., it gives the equivariant cohomology of $X_n$. Therefore $E_1^{'s,n} \cong E_1^{s,n}$, so both spectral sequences are isomorphic in the first page and we have that both spectral sequences converge to the same cohomology $H^*_G(\MM, \mathbb{R})$, thus we get the desired isomorphism between the equivariant cohomology groups. 
\end{proof}

\begin{rem}
In \cite{abadcrainic2013Bottss}, Arias Abad-Crainic introduced the notion of representations up to homotopy for Lie groupoids to extend Getzler's model of equivariant cohomology for arbitrary Lie group actions on smooth manifolds to more general actions by arbitrary Lie groupoids. Here we study generalisations of the models for equivariant cohomology of differentiable $G$-stacks $\MM$ i.e., for arbitrary Lie group actions on differentiable stacks. We can alternatively consider also the Lie groupoid $(X\times_{\MM/G}X\rightrightarrows X)$ associated to a given $G$-atlas $X$ of the differentiable $G$-stack $\MM$, whose nerve gives a simplicial smooth manifold $X_{\bullet}$. As we have seen in theorem \ref{TeoCohomologyQuotientStack} the equivariant cohomology of $\MM$ is given by the Borel construction of the fat geometric realisation of $X_{\bullet}$. We can then apply the constructions of \cite{abadcrainic2013Bottss} to the Lie groupoid $(X\times_{\MM/G}X\rightrightarrows X)$
and consider its category of representations up to homotopy and differentiable cohomology and it is an interesting question to see how they relate to the equivariant cohomology of $G$-stacks as defined here. In general, it is not yet clear if the category of representations up to homotopy is Morita invariant and therefore will give in fact an invariant of the $G$-stack as it is the case with the models discussed here. It can be expected that they will give the same equivariant cohomology in a derived setting.
In the particular case, where the differentiable $G$-stack is just a smooth manifold $X$ equipped with an action by an arbitrary Lie group $G$, we recover the classical Getzler model for non-compact Lie group actions on smooth manifolds and the associated Lie groupoid of an atlas is just the transformation groupoid $(G\times X\rightrightarrows X)$ and the models in \cite{abadcrainic2013Bottss} and ours will give the same equivariant cohomology. 
It is also interesting to extend stacky actions on differentiable stacks from Lie groups to actions of group stacks or Lie 2-groups and to develop equivariant cohomology in these more general contexts. We aim to discuss these questions in a follow-up article.
\end{rem}

\begin{rem} 
The models for equivariant cohomology of differentiable $G$-stacks described here are also closely related to the work of Safronov \cite{safronov2016hamilton} on the interpretation of quasi-Hamiltonian reduction via shifted symplectic and Poisson structures on stacks. Underlying this general framework is a symplectic stack $X$ together with a choice of a Lagrangian $L\rightarrow X$ and a $G$-stack $\MM$. A symplectic reduction is then given by a moment map $\mu: \MM/G\rightarrow X$ together with a Lagrangian structure. Shifted symplectic structures can then be detected by studying particular closed differential forms on the quotient stack giving rise to cohomology classes in the equivariant cohomology $H^\ast_G(\MM)$.
\end{rem}

\section{Spectral sequences for equivariant cohomology of differentiable \texorpdfstring{$G$}{G}-stacks}

\noindent In this final section we will derive some auxiliary results about spectral sequences for the equivariant cohomology of a differentiable $G$-stack. These will be derived in the frameworks of sheaf cohomology, hypercohomology and continuous cohomology. We also recover in special cases some spectral sequences previously constructed in the context of smooth manifolds with Lie group actions, namely by Felder-Henriques-Rossi-Zhu \cite{felder2008gerbe}, Stasheff \cite{stasheff1978continuous} and Bott \cite{bott1973chern}.

\subsection{Sheaf cohomology of quotient stacks and spectral sequences}

Let $G$ be a Lie group, $\MM$ a differentiable $G$-stack and $\mathfrak{F}$ a cartesian sheaf on the quotient stack $\mathcal{M}/G$. We refer to \cite{heinloth2005notes}, \cite{felder2008gerbe} and \cite{behrend2011differentiable} for sheaves on differentiable stacks and their cohomology. If we consider an atlas $X\rightarrow\mathcal{M}\rightarrow \MM/G$, where $X\rightarrow\mathcal{M}$ is a $G$-atlas for $\MM$ then we get an induced sheaf $\mathfrak{F}$ on $\mathcal{M}$ denoted by the same letter $\mathfrak{F}$ and also a bisimplicial sheaf $\mathfrak{F}_\bullet$ on the associated bisimplicial smooth manifold $Z_{\bullet, \bullet}$ with $Z_{p, n}= G^p \times X_n $ as described before in the proof of theorem \ref{TeoCohomologyQuotientStack}.  We can derive two spectral sequences under these assumptions. The following one relates the equivariant cohomology of the associated simplicial manifold with the equivariant cohomology of the differentiable $G$-stack itself:

\begin{teo}
There exists a spectral sequence such that
\[ E_1^{r,n} = H^r([X_n/G], \mathfrak{F}) \Rightarrow H^{r+n}_G(\mathcal{M}, \mathfrak{F}). \]
\end{teo}

\begin{proof}
We start with an acyclic equivariant  resolution $F^q$ of $\mathfrak{F}$. Then we can consider the induced sheaves $F^{p, n, q}$ on the respective slices $G^p\times X_n$. From this we get a triple complex $K^{\bullet, \bullet,\bullet}$ given by $\Gamma(G^p \times X_n, K^{p,n,q})$. Using as before the fact that the fat geometric realisation of the bisimplical smooth manifold can be computed via the fat geometric realisation of the corresponding diagonal simplicial smooth manifold, that is, 
\[\Vert p\mapsto G^p \times X_p \Vert \cong \Vert p \rightarrow \Vert n\rightarrow  G^p \times X_n  \Vert \Vert, \]
we see that the associated double complex given by $\Gamma(G^p \times X_p, K^{p,q})$ calculates the same cohomology groups. 
The resolution is acyclic and computing first with respect to $q$ we get a complex given by $\Gamma(G^p \times X_p, \mathfrak{F})$ and taking cohomology gives $H^\ast_G(\mathcal{M}, \mathfrak{F})$ as we saw in theorem \ref{TeoCohomologyQuotientStack}.
Alternatively, we consider the triple complex $E^{\bullet, \bullet, \bullet}$ defined by $\Gamma(G^p \times X_n, K^{p,q})$ with differential $d_K$ given by the resolution, the differential $\partial_G$ given as the one provided by the simplicial smooth manifold $G^\bullet$ and the third differential given by the simplicial smooth manifold $X_\bullet$. Then we can consider the associated double complex $C^{\bullet, \bullet}$ with $C^{r,n}= \bigoplus_{r=p+q}\Gamma(G^p \times X_n, K^{p,q})$ and filtering this double complex in the standard way we obtain a spectral sequence with
\[ E_0^{r,n} = \bigoplus_{r=p+q}\Gamma(G^p \times X_n, K^{p,q}).\]
Since the cohomology with respect to the index $r$ is the cohomology of the Lie groupoid $(G\times X_n \rightrightarrows X_n)$, we recover the cohomology of the quotient stack $[X_n/G]$ (see also \cite{heinloth2005notes}) and therefore
\[ E_1^{r,n} = H^r([X_n/G], \mathfrak{F}). \]
Finally, the spectral sequence of the double complex $C^{\bullet, \bullet}$ converges to the graded associated of the equivariant cohomology $H^\ast_G(\mathcal{M}, \mathfrak{F})$.
\end{proof}

If we restrict $G$ to be a countable discrete group, we get as a special case of the above spectral sequence a generalisation of a Bott type spectral sequence (see \cite{bott1973chern}) for equivariant cohomology of discrete group actions. A discrete version of the Bott spectral sequence is then obtained in the special case where the differentiable stack $\mathcal{M}$ is just a point $*$ with trivial $G$-action.

\begin{teo}
If $G$ is a countable discrete group, there exists a spectral sequence such that
\[ E_2^{p,r} = H^p(G, H^{r}(\mathcal{M}, \mathfrak{F})) \Rightarrow H^{p+r}_G(\mathcal{M}, \mathfrak{F}). \] 

\end{teo}

\begin{proof}
We consider the sheaf induced by $\mathfrak{F}$ on $\mathcal{M}$ denoted by the same letter $\mathfrak{F}$ and also the induced bisimplicial sheaf $\mathfrak{F}_\bullet$ on the associated bisimplicial smooth manifold $ Z_{\bullet,\bullet} =\{G^p \times X_n\}_{p,n\geq 0}$ as described in the proof of theorem \ref{TeoCohomologyQuotientStack}.   
Consider an acyclic resolution $K^\bullet$ of $\mathfrak{F}$ and the induced sheaves on the bisimplicial smooth manifold. We get a triple complex $F^{\bullet, \bullet, \bullet}$ given by $\Gamma(G^p \times X_n, K^{p,n,q})$. As the fat geometric realisation of the bisimplical smooth manifold is given as (compare \cite{felisattineu2007sim, ebertrandall2019semi})
\[\Vert p\mapsto G^p \times X_p \Vert \cong \Vert p \rightarrow \Vert n\rightarrow  G^p \times X_n  \Vert \Vert, \]
we see that the double complex given by $\Gamma(G^p \times X_p, K^{p,q})$ calculates the same cohomology. The resolution is acyclic and computing first with respect to $q$ we get the complex given by $\Gamma(G^p \times X_p, \mathfrak{F})$ which calculates the equivariant sheaf cohomology $H^\ast_G(\mathcal{M}, \mathfrak{F})$.

We now consider the triple complex $E^{\bullet, \bullet, \bullet}$ with $\Gamma(G^p\times X_n, K^{p,n,q})$, in which each element can be seen as a map $G^p \rightarrow \Gamma(X_n, K^{n,q})$. $G$ is discrete and for each $f(\overline{g},\overline{x})\in \Gamma(G^p\times X_n,K^{p,n,q})$, we can define now a map $\overline{f}(\overline{x})(\overline{g})=f(\overline{g},\overline{x})\in \Gamma (X_n, K^{n,q})$. We are going to consider this triple complex as a double complex with the differentials given by $\delta = d_K + (-1)^n \partial_X$ and $\partial_G$, where $d_K$ is the differential given by the resolution and $\partial_X$ induced by the simplicial structure of $X_\bullet$ and $\partial_G$ is given by the simplicial structure on $G_\bullet$. There also exists an induced action $\phi$ of $G$ on $H^n(\mathcal{M},\mathfrak{F})$ given by
\[ \phi:  G\times H^n(\mathcal{M},\mathfrak{F}) \rightarrow H^n(\mathcal{M},\mathfrak{F})  \]
\[ (g,[f])\mapsto \phi(g,[f])= [f(\overline{g}, g\cdot \overline{x})] \]
where $f\in \Gamma(G^n\times X_n, \mathfrak{F})$. This gives us as before the $E_0$-term of the spectral sequence. Hence, if we apply the differential $\delta$ first, we get therefore
\[E_1^{p,r}=C^p(G, H^r(\mathcal{M},\mathfrak{F}))\]
and if we apply the differential induced by $\partial_G$, we finally obtain for the $E_2$-term of the spectral sequence
\[ E_2^{p,r} = H^p(G, H^r(\mathcal{M},\mathfrak{F}))  \]
as we desired and the spectral sequence converges again to the associated graded of the equivariant cohomology $H^\ast_G(\mathcal{M}, \mathfrak{F})$.
\end{proof}

\begin{exm}
If the differentiable stack $\MM$ is just a smooth manifold $X$ with an action by a discrete group $G$, the above result gives a spectral sequence of the form
\[ E_2^{p,n} = H^p(G, H^{n}(X, \mathfrak{F})) \Rightarrow H^{p+n}([X/G], \mathfrak{F}) \]
which was previously also obtained by Felder-Henriques-Rossi-Zhu (see \cite[A.4]{felder2008gerbe}).
\end{exm}

\subsection{Hypercohomology of quotient stacks and spectral sequences}

Let $G$ be a Lie group and $\mathcal{M}$ a differentiable $G$-stack. Furthermore, let $\mathfrak{F}_0 \rightarrow \mathfrak{F}_1 \rightarrow \cdots \rightarrow \mathfrak{F}_m $ be a complex of cartesian sheaves of abelian groups on $\mathcal{M}/G$ with an atlas given by $X\rightarrow \MM \rightarrow \MM/G$, where $X\rightarrow \MM$ is a $G$-atlas for $\mathcal{M}$. 
We refer again to \cite{heinloth2005notes}, \cite{felder2008gerbe} and \cite{behrend2011differentiable} for sheaves and complexes of sheaves on differentiable stacks and their cohomologies. For any $r$, let $\mathfrak{F}_r$ be the sheaf from the complex on $\mathcal{M}/G$ and $\mathfrak{F}_{r,\bullet}$ be the induced bisimplicial sheaf on the associated bisimplicial smooth manifold $G^\bullet \times X_\bullet$. As above we can relate the equivariant hypercohomology of the nerve with the equivariant hypercohomology of the stack as follows

\begin{teo}
There exists a spectral sequence such that
\[ E_1^{s,n} = H^{s}([X_n/G], \mathfrak{F}_0 \rightarrow \mathfrak{F}_1 \rightarrow\cdots \rightarrow \mathfrak{F}_m) \Rightarrow H^{s+n}_G(\mathcal{M}, \mathfrak{F}_0 \rightarrow \mathfrak{F}_1 \rightarrow \cdots \rightarrow \mathfrak{F}_m). \]
\end{teo}

\begin{proof}
Let $K^{q}_r$ be an acyclic equivariant resolution of $\mathfrak{F}_r$ and denote by $K^{p,n,q}_r$ the sheaves induced on $G^p \times X_n$. Then we have the quadruple complex $C^{\bullet, \bullet, \bullet, \bullet}$ with
\[ C^{p,n,q,r}=\Gamma(G^p \times X_n, K^{p,n,q}_r). \] 
As before we can consider the triple complex $C^{\bullet, \bullet, \bullet}$ given by $C^{p,q,r}= \Gamma(G^p \times X_p, K^{p,q}_r)$  and we see that $K^{p,q}_r$ is acyclic. Therefore, if we compute with respect to the index $q$ we get the double complex $C^{\bullet, \bullet}$ with  $C^{p,r}= \Gamma(G^p \times X_p, \mathfrak{F}^{p}_r)$. Hence the quadruple complex has as cohomology $H^*_G(\mathcal{M},\mathfrak{F}_0 \rightarrow \mathfrak{F}_1 \rightarrow \cdots \rightarrow \mathfrak{F}_m )$.

On the other hand, we can interpret also the quadruple complex as a double complex if we consider the complex given by $ C^{s,n}= \bigoplus_{s=p+q+r}\Gamma(G^p \times X_{n}, K^{p, q}_r) $ with its respective differential and as a second differential the one given on $X_\bullet$. Then we get a spectral sequence with
\[ E_0^{s,n} = \bigoplus_{s=p+q+r}\Gamma(G^p \times X_{n}, K^{p,q}_r)\]
and since the cohomology with respect to the index $s$ is the cohomology of the Lie groupoid $(G\times X_n \rightrightarrows X_n)$, we obtain the cohomology of the quotient stack $[X_n/G]$ (see \cite{behrend2011differentiable} and \cite{heinloth2005notes}). Therefore we have
\[ E_1^{s,n} = H^{s}([X_n/G], \mathfrak{F}_0 \rightarrow \mathfrak{F}_1 \rightarrow\cdots \rightarrow \mathfrak{F}_m) \]
and the spectral sequence of the double complex converges in the standard way to the associated graded of the equivariant hypercohomology
$H^*_G(\mathcal{M}, \mathfrak{F}_0 \rightarrow \mathfrak{F}_1 \cdots \rightarrow \rightarrow\mathfrak{F}_m)$
of the differentiable $G$-stack $\MM$. 
\end{proof}

If we again restrict the discussion above to a countable discrete group, then we obtain as a special case the following spectral sequence

\begin{teo}
If $G$ is a countable discrete group, there exists a spectral sequence such that
\[ E_2^{p,s} = H^p(G,H^s(\mathcal{M}, \mathfrak{F}_0 \rightarrow \mathfrak{F}_1 \rightarrow\cdots \rightarrow \mathfrak{F}_m)) \Rightarrow H^{p+s}_G(\mathcal{M}, \mathfrak{F}_0 \rightarrow \mathfrak{F}_1 \cdots \rightarrow \mathfrak{F}_m). \]
\end{teo}

\begin{proof}

As before, let $K^{q}_r$ be an acyclic equivariant resolution of $\mathfrak{F}_r$ and denote by $K^{p,n,q}_r$ the induced sheaves on $G^p \times X_n$. Then we get a quadruple complex $C^{\bullet, \bullet, \bullet, \bullet}$ given by
\[ C^{p,n,q,r}=\Gamma(G^p \times X_n, K^{p,n,q}_r). \] 
As before we can consider the triple complex $C^{\bullet, \bullet, \bullet}$ with $C^{p,q,r}= \Gamma(G^p \times X_p, K^{p,q}_r)$. If we compute with respect to the resolution we get again a double complex $C^{\bullet, \bullet}$ with $C^{p,r}= \Gamma(G^p \times X_p, \mathfrak{F}^{p}_r)$. Hence the quadruple complex has as cohomology $H^*_G(\mathcal{M},\mathfrak{F}_0 \rightarrow \mathfrak{F}_1 \rightarrow \cdots \rightarrow \mathfrak{F}_m )$.

We can also consider the quadruple complex with $C^{p,n,q,r}$ as a double complex given by the differential for the indices $n, q, r$ and as the second differential the one given by the simplicial structure of $G_\bullet$. Moreover if we consider the elements in $C^{p,n,q,r}$ as maps $G^p \rightarrow \Gamma(X_n, K^q_r)$, we see that for the filtration on the induced double complex we have
\[ E_0^{p,s} = C^p(G, \bigoplus_{n+q+r=s}\Gamma(X_n, K^q_r))\]
\[ E_1^{p,s} = C^p(G, H^s(\mathcal{M}, \mathfrak{F}_0 \rightarrow \mathfrak{F}_1 \rightarrow\cdots \rightarrow \mathfrak{F}_m ))\]
and finally for the $E_2$-term of the spectral sequence we get
\[ E_2^{p, s} = H^p(G, H^s(\mathcal{M}, \mathfrak{F}_0 \rightarrow \mathfrak{F}_1 \rightarrow\cdots \rightarrow \mathfrak{F}_m )). \]
As before, the spectral sequence converges to the associated graded of the equivariant hypercohomology of the differentiable $G$-stack $\MM$. 

\end{proof}

\begin{exm}
If $\MM$ is again a smooth manifold $X$ with an action of a discrete group $G$, the previous spectral sequence generalises the spectral sequence
\[ E_2^{p,n} = H^p(G,H^n(X, \mathfrak{F}_0 \rightarrow \mathfrak{F}_1 \rightarrow\cdots \rightarrow \mathfrak{F}_m)) \Rightarrow H^{p+n}([X/G], \mathfrak{F}_0 \rightarrow \mathfrak{F}_1 \cdots  \rightarrow\mathfrak{F}_m) \]
which was constructed also in \cite[A.7]{felder2008gerbe}.
\end{exm}

\subsection{Spectral sequences for continuous cohomology of Lie group actions}

Let $G$ be a Lie group and $\mathcal{M}$ a differentiable $G$-stack. We then get the following generalised version of Bott's spectral sequence, where $H^*_c(-)$ denotes continuous cohomology (compare \cite{bott1973chern}, \cite{stasheff1978continuous}).

\begin{teo}
There exists a spectral sequence such that 
\[ E_1^{k,p} = \bigoplus_{q+n=k} H ^p_c(G, \bigoplus_{s+t=q} \Omega_{dR}^s(X_n)\otimes S^t(\mathfrak{g}^\vee))\Rightarrow H^{k+p}_{G}(\mathcal{M}, \mathbb{R}) \]
and where
\[ E_2^{k,p} = \text{Tot}\bigoplus_{q+n=k} H ^p_c (G, \bigoplus_{s+t=q} \Omega_{dR}^s(X_n)\otimes S^t(\mathfrak{g^\vee}) ). \]
\end{teo}

\begin{proof}
Let $\Omega_{dR}^{\bullet}(G^{\bullet} \times X_{\bullet})$ be the triple complex given by $(\Omega_{dR}^q(G^p\times X_n), d_{dR}, \partial_G,\partial_X)$ which computes the de Rham cohomology $H^{*}_{dR}(\mathcal{M}/G)$ of the quotient stack. We also consider the complex of smooth functions as discussed in \cite{heller1973principal}, \cite{hu1952cohomology} and \cite{stasheff1978continuous} given by
\[C_\infty (G^p, \bigoplus_{s+t=q} \Omega_{dR}^s(X_n)\otimes \wedge^t(\mathfrak{g}^{\vee p}) )\]
with differentials induced by $d_{dR}$, $\partial_G $ and $\partial_X$. So if we take the function
\[ \Psi: \Omega_{dR}^q(G^p\times X_n) \rightarrow C_\infty (G^p, \bigoplus_{s+t=q} \Omega_{dR}^s(X_n)\otimes \wedge^t(\mathfrak{g}^{\vee p}) )  \]
\[ \sum \omega_i(\vec{g},\vec{x})dg_I dx_J \mapsto \vec{g}\xrightarrow{\omega} (\omega_i(\vec{g},\vec{x}) dx_J \otimes dg_I),\]
which commutes with the differentials and is a bijection, we see that the complex of smooth functions computes the same cohomology as the complex given by $\Omega_{dR}^q(G^p\times X_n)$.

Moreover, if we consider the double complex $C^{\bullet, \bullet}=\{C^{k, p}\}$ given by
\[ C^{k,p} = \bigoplus_{q+n=k} C_\infty (G, \bigoplus_{s+t=q} \Omega_{dR}^s(X_n)\otimes \wedge ^t(\mathfrak{g}^{\vee p}) ), \]
then we get for the spectral sequence  $E^{k,p}_0 = C^{k,p}$ and therefore
\[E^{k,p}_1 = \bigoplus_{q+n=k } H_\infty^p(G,\bigoplus_{s+t=q} \Omega_{dR}^s(X_n)\otimes H^p (\wedge ^t(\mathfrak{g}^{\vee p})) )\]
\[= \bigoplus_{q+n=k} H_\infty ^p (G, \bigoplus_{s+t=q} \Omega_{dR}^s(X_n)\otimes S^t(\mathfrak{g}^\vee))\]
Here $H_\infty ^*(-)$ means smooth cohomology with respect to smooth cochains. However, these smooth cochains can be approximated by continuous ones (see \cite[5.]{hochschild1962cohomology}, \cite[6.]{stasheff1978continuous}, \cite[I.]{wagemann2015cocycle}). Therefore we can consider this cohomology as the continuous cohomology $H ^*_c(-)$, (compare also \cite{hu1952cohomology}). Hence we get for the second term of the spectral sequence
\[ E^{k,p}_2=\text{Tot} \bigoplus_{q+n=k} H^p_c (G, \bigoplus_{s+t=q} \Omega_{dR}^s(X_n)\otimes S^t(\mathfrak{g}^{\vee }))  \]
and the spectral sequence of the double complex convergences to the equivariant cohomology $H^*_{G}(\mathcal{M}, \mathbb{R})$ of the differentiable $G$-stack $\MM$ as desired.
\end{proof}

\begin{exm}
If the differentiable stack $\MM$ is just a point $*$ with trivial $G$-action, the above spectral sequence reduces to the following spectral sequence converging to the cohomology of the classifying space of the Lie group $G$
\[ E^{t,p}_1= H^p_c(G, S^t(\mathfrak{g}^\vee)) \Rightarrow H^{t+p}(BG, \mathbb{R}) \]
as the homotopy type of the classifying stack $\mathcal{B} G=[*/G]$ is given by the classifying space $BG$ of $G$. This is Bott's spectral sequence as originally constructed by Bott in \cite{bott1973chern} (compare also \cite{stasheff1978continuous}). If $G$ is in addition compact, Bott's spectral collapses and recovers the classical Borel isomorphism $H^*(BG, \mathbb{R})\cong S^*(\mathfrak{g}^\vee)^G$.
\end{exm}

\begin{rem}
In \cite{abadcrainic2013Bottss}, Arias Abad-Crainic extended Bott's spectral sequence to a spectral sequence converging to the cohomology of the classifying space of a general Lie groupoid, which generalises also a similar spectral sequence for the case of flat groupoids constructed previously by Behrend \cite{behrend2005deRham}. Related spectral sequences in the special situation of manifolds with Lie group actions were also developed more recently by Arias Abad-Uribe \cite{ariasabaduribe2015ss} and Garcia-Compean-Paniagua-Uribe \cite{garciacompean2014equi}. These are concerned in particular with actions of non-compact Lie groups on smooth manifolds generalising Bott's original spectral sequence. This corresponds to our situation above for the case where the differentiable stack $\MM$ is a smooth manifold equipped with an action by an arbitrary Lie group $G$.\\
\end{rem}

\noindent {\it Acknowledgements:} The first author was supported by the Colciencias Conv. 646 scholarship of the Colombian government. Both authors like to thank the referee for helpful suggestions, comments and corrections to improve this article.

\end{document}